\documentclass[12pt, a4paper]{amsart}

\usepackage{amsfonts,amsmath,amssymb,amscd}

\newtheorem{theorem}{Theorem}[section]
\newtheorem{lemma}[theorem]{Lemma}
\newtheorem{prop}[theorem]{Proposition}
\newtheorem{corollary}[theorem]{Corollary}

\theoremstyle{definition}
\newtheorem{definition}[theorem]{Definition}
\newtheorem{rem}[theorem]{Remark}

\newtheorem{example}[theorem]{Example}

\usepackage{color}

\newcommand\pf{\begin{proof}}
\newcommand\epf{\end{proof}}

\newcommand{\cmdblackltimes}{\mathop{\raisebox{0.2ex}{\makebox[0.92em][l]{${\scriptstyle\blacktriangleright\mathrel{\mkern-4mu}<}$}}}}

\numberwithin{equation}{section}

\begin{document} 

\title[Mini-course on Hopf algebras]
{Mini-course on Hopf algebras\\
---Hopf crossed products---}

\author[Akira Masuoka]{Akira Masuoka}
\address{Akira Masuoka: 
Institute of Mathematics, 
University of Tsukuba, 
Ibaraki 305-8571, Japan}
\email{akira@math.tsukuba.ac.jp}


\maketitle



\section*{Introduction}

Hopf crossed products, or in other words, cleft comodule algebras
form a special but important class in Hopf-Galois extensions. To discuss this
interesting subject, we will start with the more familiar group crossed products,
and then see that they are naturally generalized by Hopf crossed products;
these Hopf crossed products are characterized as Hopf-Galois extensions with normal basis; see 
Theorem \ref{thm:DT}.  
After showing this characterization due to Doi and Takeuchi, we will proceed 
to two applications of Hopf crossed products---the equivariant
smoothness of Hopf algebras, and
the tensor product decomposition in super-commutative Hopf superalgebras; see Theorems 
\ref{thm:equivariant_smoothness}, \ref{thm:tensor_product_decomposition}. 

\vspace{3mm}
For simplicity we work over a field, though all the results in Sections 1, 2 hold true
over any commutative ring. Let $k$ denote the field over which we work. 
Vector spaces, algebras, coalgebras
and Hopf algebras are those over $k$. The unadorned tensor product $\otimes$ is taken over $k$. 
All algebras are supposed to be non-zero.

\section{Group crossed products}\label{sec:group}

In what follows, $\Gamma$ denotes a group, whose identity element is denoted by 1. Let $A$ be a
$\Gamma$-graded algebra. Thus, $A$ decomposes as a direct sum 
\[ A = \bigoplus_{g\in \Gamma} A_g \]
of sub-vector spaces $A_g$, $g \in \Gamma$, so that 
\[ 1 \in A_1, \quad A_gA_h\subset A_{gh}, \quad g, h \in \Gamma. \]
Set $B = A_1$, the neutral component. Then each component $A_g$ is a $B$-bimodule, and the product
maps give $B$-bimodule maps, 
\[ \mu_{g,h} : A_g \otimes_B A_h \to A_{gh},\quad g, h \in \Gamma. \]
We say that $A$ is \emph{strongly graded}, if $A_gA_h = A_{gh}$, or namely if $\mu_{g,h}$ are surjective
for all $g, h \in \Gamma$. 

\begin{lemma}\label{lem:strongly_graded}
Let $A = \bigoplus_{g\in \Gamma} A_g$ be a $\Gamma$-graded algebra with $B = A_1$, as above. 

(1) Fix $g \in \Gamma$. The $B$-bimodules $A_g$, $A_{g^{-1}}$ together with the $B$-bimodule maps
\[ \mu_{g, g^{-1}} : A_g \otimes_B A_{g^{-1}} \to B, \quad 
\mu_{g^{-1}, g} : A_{g^{-1}} \otimes_B A_g \to B\] 
form a Morita context \cite[p.527]{R}. 

(2) TFAE:
\begin{itemize}
\item[(a)] $A$ is strongly graded;
\item[(b)] For each $g \in \Gamma$, the Morita context obtained above is strict, i.e., 
$\mu_{g,g^{-1}}$ and $\mu_{g^{-1},g}$ are both surjective; 
\item[(c)] $\mu_{g,h}$ are bijective for all $g, h \in \Gamma$. 
\end{itemize}
\end{lemma}
\begin{proof}
(1) This follows by the associativity of the product on $A$. 

(2) (a) $\Rightarrow$ (b). Condition (a) is implies that $\mu_{g,g^{-1}}$, $\mu_{g^{-1},g}$ are surjective,
which is equivalent to (b).

(b) $\Rightarrow$ (c). Choose $g, h \in \Gamma$ arbitrarily, and assume (b). 
Then the Morita Theorem \cite[Remark~25A.16, p.528]{R} shows that $\mu_{g,g^{-1}}$ is
bijective. From the commutative diagram
\begin{equation*}
\begin{CD}
A_g \otimes_B A_{g^{-1}} \otimes_B A_{gh}  @>{\mathrm{id}\otimes \mu_{g^{-1},gh}}>> A_g\otimes_B A_h \\
@V{\mu_{g,g^{-1}}\otimes \mathrm{id}}VV @VV{\mu_{g,h}}V \\
B\otimes_B A_{gh} @>{\simeq}>> A_{gh}\, ,
\end{CD}
\end{equation*}
we see that $\mathrm{id}\otimes \mu_{g^{-1},gh}$ is injective, and $\mu_{g,h}$ is surjective. The last
surjectivity applied to $\mu_{g^{-1},gh}$ shows that the two maps are both bijective. 

(c) $\Rightarrow$ (a). Obvious. 
\end{proof}

We are going to 
construct \emph{group crossed products}, which form a special class of strongly graded algebras.
Let $B$ be an algebra. Suppose that for each $g \in \Gamma$, there is given an algebra
endomorphism $g \rightharpoonup : B \to B$, $b \mapsto g\rightharpoonup b$. In addition,
let $\sigma : \Gamma \times \Gamma \to B^{\times}$ be a map with values in the group $B^{\times}$
of all invertible elements in $B$, and suppose that the following conditions are satisfied.
\begin{equation}\label{G-data1}
1 \rightharpoonup b = b, \quad \sigma(g, 1) = 1 = \sigma(1, g), 
\end{equation}
\begin{equation}\label{G-data2}
[g \rightharpoonup(h\rightharpoonup b)]\, \sigma(g, h) 
= \sigma(g, h)\, (g h \rightharpoonup b), 
\end{equation}
\begin{equation}\label{G-data3}
[g \rightharpoonup\sigma(h, \ell )]\, \sigma(g, h \ell )
= \sigma(g, h )\, \sigma(g h, \ell ), 
\end{equation}
where $b \in B$, $g, h, \ell \in \Gamma$. Suppose $h = g^{-1}$ in \eqref{G-data2}. Since one then sees that
the composite of $g \rightharpoonup$ with $g^{-1} \rightharpoonup$ is an inner automorphism on $B$, it
follows that each $g \rightharpoonup$ gives an automorphism on $B$. 

\begin{lemma}\label{lem:group_crossed_product}
Let $A =\bigoplus_{g\in \Gamma} Bu_g$ be the free left $B$-module with basis $u_g$, $g \in \Gamma$. 
Define a product on $A$ by
\[ (b \, u_g)(c \, u_h) = b(g \rightharpoonup c)\, 
\sigma(g, h) \, u_{gh}, \]
where $b, c \in B$, $g, h\in \Gamma$. Then $A$ turns into a $\Gamma$-graded algebra with $A_g =Bu_g$. 
This $A$ has $u_1$ as its unit, and includes $B$ as the neutral component. Moreover, this is strongly graded. 
\end{lemma}
\begin{proof}
Condition \eqref{G-data1} ensures that $u_1$ is a unit, and Conditions \eqref{G-data2} and \eqref{G-data3}
ensure
\[ u_g(u_hb) =(u_gu_h)b, \quad u_g(u_hu_{\ell}) = (u_gu_h)u_{\ell}, \]
where $b \in B$, $g, h, \ell \in \Gamma$, respectively. It follows that $A$ is an algebra, which is easily 
seen to be $\Gamma$-graded with $A_1 = B$. Since 
\[ u_{gh} = \sigma(g,h)^{-1}u_g u_h \in BA_g A_h=A_gA_h, \]
it follows that $A_{gh} = Bu_{gh} \subset A_g A_h$, whence $A$ is strongly graded.   
\end{proof}

\begin{definition}\label{def:group_crossed_product}
The strongly $\Gamma$-graded algebra constructed above is denoted by 
$B \rtimes_{\sigma} \Gamma$ (with $\rightharpoonup$
omitted), and is called the $\Gamma$-\emph{crossed product} associated with $\rightharpoonup$, $\sigma$. 
\end{definition}

The following characterizes $\Gamma$-crossed products in (strongly) $\Gamma$-graded algebras. 

\begin{theorem}\label{thm:group_crossed_product}
For a $\Gamma$-graded algebra $A = \bigoplus_{g\in \Gamma} A_g$ with $B = A_1$, 
TFAE:
\begin{itemize}
\item[(a)] Each component $A_g$ contains an invertible element in $A$; 
\item[(b)] $A$ is strongly graded, and each component $A_g$ is isomorphic to $B$ as a left $B$-module;
\item[(c)] There is a $\Gamma$-graded algebra isomorphism $\alpha$ from $A$ to some 
$\Gamma$-crossed product
$B \rtimes_{\sigma} \Gamma$ such that $\alpha|_{B} =\mathrm{id}_B$,  the identity map on $B$. 
\end{itemize}
\end{theorem}
\begin{proof}
(c) $\Rightarrow$ (b). Obvious. 

(b) $\Rightarrow$ (a). Assume (b), and let $v_g \in A_g$ be a left $B$-free basis element, $A_g =B v_g$. We
claim that $v_g$ is invertible in $A$. By Lemma \ref{lem:strongly_graded}(2), $\mu_{g^{-1}, g}$ and the map
\[ A_{g^{-1}} \to \operatorname{Hom}_{B-} (A_g, B), \ y \mapsto (x \mapsto xy) \]
are both bijective; see \cite[Lemma~25A.17(i), p.528]{R}. 
The first (resp., second) bijection shows that $v_g$ has a left (resp., right) inverse. 

(a) $\Rightarrow$ (c). Assume (a). For each $g\in \Gamma$, choose an invertible element $u_g \in A_g$,
so that $u_1 = 1$, the unit of $A$. Since it follows that $u_g^{-1} \in A_{g^{-1}}$, we have $A_gu_g^{-1} \subset B$,
whence $A_g = Bu_g$. 
 Define $g \rightharpoonup : B \to B$, $\sigma : \Gamma \times \Gamma \to B^{\times}$ by
\[ g\rightharpoonup b = u_gb\, u_g^{-1}\ (b\in B), \quad \sigma(g,h) = u_gu_hu_{gh}^{-1}. \]
Then one sees that these satisfy \eqref{G-data1}--\eqref{G-data3}, and $A = B\rtimes_{\sigma} \Gamma$. 
\end{proof}

\section{Hopf crossed products}\label{sec:Hopf}

Let $C$ be a coalgebra. The coproduct on $C$ is denoted by 
$$ \Delta : C \to C\otimes C,\ \Delta(c) = c_{(1)} \otimes c_{(2)} $$
(cf. \cite[p.10]{Sw}), and the counit is denoted by $\varepsilon : C \to k$. 
Let $A$ be an algebra. The vector space $\operatorname{Hom}(C, A)$ of all linear maps $C\to A$
form an algebra with respect to the product $*$ \cite[p.69]{Sw} defined by
\begin{equation}\label{*-product}
f * g (c) = f(c_{(1)})\, g(c_{(2)}), \quad c \in C, 
\end{equation}
where $f, g \in \operatorname{Hom}(C, A)$. The unit is given by $c \mapsto \varepsilon(c)1$. We say that $f$
is $*$-\emph{invertible} if it is invertible with respect to the product $*$. 

Recall that a \emph{bialgebra} is an algebra and coalgebra $H=(H, \Delta, \varepsilon)$ 
such that $\Delta : H \to H\otimes H$, $\varepsilon : H \to k$ are algebra maps.
A bialgebra $H$ is called a \emph{Hopf algebra} if 
the identity map $\mathrm{id}_H$ on $H$ is $*$-invertible. The inverse $S$ of $\mathrm{id}_H$ is called
the \emph{antipode}, and is characterized as a linear map $S : H \to H$ such that 
\[ h_{(1)}S(h_{(2)}) = \varepsilon(h)1 = S(h_{(1)})h_{(2)}, \quad h \in H. \]  

\begin{example}\label{ex:group_algebra1}
Let $\Gamma$ be a group. The group algebra $k\Gamma$ naturally forms a Hopf algebra with respect to the 
structure defined by
\[ \Delta(g) = g \otimes g, \quad \varepsilon(g) = 1, \quad S(g) = g^{-1}, \]
where $g \in \Gamma$.
\end{example}

In what follows, $H=(H, \Delta, \varepsilon)$ denotes a Hopf algebra. A right $H$-\emph{comodule algebra}
\cite[p.40]{Mon} is an algebra $A$ given an algebra map $\rho_A : A \to A \otimes H$ with 
which $A = (A, \rho_A)$ is a 
right $H$-comodule. The definition of \emph{left} $H$-\emph{comodule algebras} is obvious. 
By saying an $H$-comodule algebra we mean a right $H$-comodule algebra.
Let $A$ be an $H$-comodule algebra. Define
\begin{equation}\label{coinvariants}
A^{\mathrm{co}H} :=\{ a \in A \mid \rho_A(a) = a \otimes 1~\text{in}~A \otimes H \}. 
\end{equation}
This forms a subalgebra of $A$, and is called the \emph{subalgebra of} $H$-\emph{coinvariants} \cite[p.13]{Mon}. 
Set $B = A^{\mathrm{co}H}$. Then $A$ is called an $H$-\emph{Galois extension over}
$B$ \cite[p.123]{Mon}, if the (well-defined) left $A$-linear map
\[ \beta_A : A \otimes_B A \to A \otimes H,\ a \otimes_B a' \mapsto a\, \rho_A(a') \]
is bijective. 

\begin{example}\label{ex:group_algebra2}
A graded $\Gamma$-graded algebra is the same as a $k\Gamma$-comodule algebra. In fact, 
if $A = \bigoplus_{g \in \Gamma}A_g$ is a $\Gamma$-graded algebra, 
then it turns into a $k\Gamma$-comodule algebra
by defining $\rho_A(a) = a \otimes g$ for $a \in A_g$. 
Conversely, a $k\Gamma$-comodule algebra structure
$\rho_A$ on $A$ makes $A$ into a $\Gamma$-graded algebra with $A_g = \{ a \in A\mid \rho_A(a) =a \otimes g\}$. 
These constructions are inverses of each other. 
The neutral component $A_1$ coincides with $A^{\mathrm{co}k\Gamma}$. Recall from Lemma 
\ref{lem:strongly_graded}(b) that a $\Gamma$-algebra $A$ with neutral component $B$ is strongly graded
if and only if the product maps $A_g\otimes_B A_h \to A_{gh} (= A_{gh} \otimes h)$ are bijective for all $g, h \in
\Gamma$. Taking the direct sum $\bigoplus_{g,h \in \Gamma}$, we see that the last condition is equivalent
to that the map
\[ A \otimes_B A \to A \otimes k\Gamma, \quad 
\sum_g a_g \otimes \sum_h a'_h \mapsto \sum_{g,h} a_ga'_h \otimes h, \] 
where $a_g \in A_g$, $a'_h \in A_h$, 
is bijective. Since the last map coincides with $\beta_A$, it follows that a strongly graded $\Gamma$-algebra 
with neutral component $B$ is 
the same as a $k\Gamma$-Galois extension over $B$.
\end{example} 

Let us generalize the construction of group crossed products, replacing groups with Hopf algebras. 
Let $B$ be an algebra. A linear map $\rightharpoonup : H \otimes B \to B$ is called 
a \emph{measuring} \cite[p.137]{Sw}, if 
\[ h \rightharpoonup 1 = \varepsilon(h) 1\quad \text{and}\quad h \rightharpoonup bc  
= (h_{(1)} \rightharpoonup b)(h_{(2)} \rightharpoonup c),~~b, c \in B. \]
A pair $(\rightharpoonup, \sigma)$ of a measuring $\rightharpoonup : H \otimes B \to B$ and
a $*$-invertible linear map $\sigma : H \otimes H \to B$ is called a \emph{crossed system} 
(for $H$ over $B$) \cite[p.3055]{D},
if the following conditions are satisfied:   
\begin{equation}\label{H-data1}
1 \rightharpoonup b = b, \quad \sigma(h, 1) = \varepsilon(h)1 = \sigma(1, h), 
\end{equation}
\begin{equation}\label{H-data2}
[g_{(1)}\rightharpoonup(h_{(1)}\rightharpoonup b)]\, \sigma(g_{(2)}, h_{(2)}) 
= \sigma(g_{(1)}, h_{(1)})\, (g_{(2)}h_{(2)}\rightharpoonup b), 
\end{equation}
\begin{equation}\label{H-data3}
[g_{(1)}\rightharpoonup\sigma(h_{(1)}, \ell_{(1)})]\, \sigma(g_{(2)}, h_{(2)}\ell_{(2)})
= \sigma(g_{(1)}, h_{(1)})\, \sigma(g_{(2)}h_{(2)}, \ell ), 
\end{equation}
where $b \in B$, $g, h, \ell \in H$. One verifies directly the following.

\begin{lemma}\label{lem:H-crossed_product}
Given a crossed system $(\rightharpoonup, \sigma)$,
the right $H$-comodule $B \otimes H= (B \otimes H, \mathrm{id}_B \otimes \Delta)$ 
turns into an $H$-comodule algebra with respect to the product
defined by
\[ (b \otimes g)(c \otimes h) = b(g_{(1)}\rightharpoonup c)\, 
\sigma(g_{(2)}, h_{(1)}) \otimes g_{(3)}h_{(2)}, \]
where $b, c \in B$, $g, h\in H$. This $H$-comodule algebra has $1\otimes 1$ as its unit, and includes
$B = B\otimes k$ as the subalgebra of $H$-coinvariants.  
\end{lemma}

\begin{definition}\label{def:H-crossed_product}
The $H$-comodule algebra constructed above is denoted by $B \rtimes_{\sigma} H$ (with $\rightharpoonup$
omitted), and is called the $H$-\emph{crossed product} associated with $(\rightharpoonup, \sigma)$. 
\end{definition}

\begin{rem}\label{rem:smash}
Suppose that a measuring $\rightharpoonup : H \otimes B \to B$ is a left $H$-module structure on $B$, 
or in other words that $B$ is a left $H$-\emph{module algebra} \cite[p.153]{Sw}
under $\rightharpoonup$. Then this 
$\rightharpoonup$ together with the trivial $\sigma : H \otimes H \to B$, $\sigma(g, h) = \varepsilon(g)
\varepsilon(h)$ form a crossed system. In this case, $B \rtimes_{\sigma} H$ 
is denoted by $B \rtimes H$, and is called the $H$-\emph{smash product} \cite[p.155]{Sw} associated with
$\rightharpoonup$. The product on $B \rtimes H$ is given by
\[ (b \otimes g)(c \otimes h) = b(g_{(1)}\rightharpoonup c) \otimes g_{(2)}h. \]
\end{rem} 

\begin{example}\label{ex:group_algebra3}
Suppose that $H$ is a group algebra $k\Gamma$. Giving a measuring $k\Gamma \otimes B \to B$
is the same as giving an algebra endomorphism $g \rightharpoonup : B \to B$ for each $g \in \Gamma$. 
Giving an invertible linear map $k\Gamma \otimes k\Gamma \to B$ is the same as giving an element
$\sigma(g, h)$ in $B^{\times}$ for each $g, h \in H$. One sees, moreover, that the $k\Gamma$-crossed 
products $B \rtimes_{\sigma} k\Gamma$ constructed above coincide with the $\Gamma$-crossed 
products $B \rtimes_{\sigma} \Gamma$ constructed in the preceding section. 
\end{example}

\begin{definition}\label{def:cleft}
An $H$-comodule algebra $A$ is said to be \emph{cleft} if there is a $*$-invertible
$H$-comodule map $\phi : H \to A$. We remark that if this is the case, $\phi$ can be chosen so that
$\phi(1) = 1$, by replacing $\phi$ with 
\begin{equation}\label{replacement}
h \mapsto \phi^{-1}(1)\, \phi(h). 
\end{equation}
An invertible $H$-comodule map $\phi: H \to A$ such that $\phi(1) = 1$ is called a \emph{section}. 
\end{definition}

\begin{theorem}[\cite{D}, \cite{DT}]\label{thm:DT} 
For an $H$-comodule algebra $A$ with $B = A^{\mathrm{co}H}$, 
TFAE:
\begin{itemize}
\item[(a)] $A$ is cleft; 
\item[(b)] $A$ is an $H$-Galois extension over $B$, and there is a left $B$-linear and right $H$-colinear
isomorphism $A \simeq B \otimes H$; 
\item[(c)] There is an $H$-comodule algebra isomorphism $\alpha$ from $A$ to some $H$-crossed product
$B \rtimes_{\sigma} H$ such that $\alpha|_{B} =\mathrm{id}_B$. 
\end{itemize}
\end{theorem} 

\begin{rem}\label{rem:compare_with_graded_algebra}
Let $A$ be a $\Gamma$-graded algebra with $B = A_1$, or in other words, a $k\Gamma$-comodule algebra
with $B = A^{\mathrm{co}k\Gamma}$. One sees
that $A$ is cleft if and only if each component $A_g$ contains a unit in $A$, and that 
there is a left $B$-linear and right $k\Gamma$-colinear isomorphism $A \simeq B \otimes H$ if
and only if each component $A_g$ is isomorphic to $B$
 as a left $B$-module. This together with Examples \ref{ex:group_algebra2}, \ref{ex:group_algebra3}
conclude that Theorem \ref{thm:DT} generalizes Theorem \ref{thm:group_crossed_product}. 
\end{rem}

\begin{proof}[Sketchy Proof of Theorem \ref{thm:DT}] 
(a) $\Rightarrow$ (b). Let $\phi: H \to A$ be a section. Then one sees that
\[ A \otimes H \to A \otimes_B A, \ a \otimes h \mapsto a\, \phi^{-1}(h_{(1)}) \otimes_B \phi(h_{(2)}) \]
gives an inverse of $\beta_A$, whence $A$ is an $H$-Galois extension over $B$. One also sees that
\[ B \otimes H \to A,\ b \otimes h \mapsto b\, \phi(h) \]
is left $B$-linear and right $H$-colinear map, and is indeed an isomorphism with inverse
\[ A \to B \otimes H,\ a \mapsto a_{(0)}\phi^{-1}(a_{(1)}) \otimes a_{(2)}, \]
where
$a_{(0)}\otimes a_{(1)}\otimes a_{(2)}= (\rho_A\otimes \mathrm{id}_H)\circ \rho_A(a) 
(= (\mathrm{id}_A\otimes \Delta)\circ \rho_A(a))$.   

(b) $\Rightarrow$ (a). Assume (b). 
Let $\operatorname{Hom}_{B-}^{-H}$ (resp., $\operatorname{End}_{B-}^{-H}$)
denote the vector space (resp., the algebra) of left $B$-linear and right 
$H$-colinear maps (resp., endomorphisms). Let  
\[ \xi : \operatorname{Hom}_{B-}^{-H}(B \otimes H, A) \to \operatorname{Hom}(H, A), \]
be the linear map given by $f \mapsto (h \mapsto f(1 \otimes h))$.  
Let 
\[ \eta : \operatorname{Hom}_{B-}^{-H}(A, B \otimes H) \to \operatorname{Hom}(H, A) \]
be the composite of the linear map $\operatorname{Hom}_{B-}^{-H}(A, B \otimes H) \to
\operatorname{End}_{B-}(A)$ given by $g \mapsto (a \mapsto (\mathrm{id}_B \otimes \varepsilon) \circ g(a))$, 
with the isomorphism $\operatorname{Hom}_{B-}^{-H}(\beta_A^{-1}, A) :
\operatorname{End}_{B-}^{-H}(A)\overset{\simeq}{\longrightarrow} \operatorname{Hom}^{-H}(H, A)\,
(\subset \operatorname{Hom}(H, A))$. 
We see that $\xi$, $\eta$ are both injective. We claim that $\xi$ restricts to an injection
\[ \operatorname{Isom}_{B-}^{-H}(A, B \otimes H) \to \operatorname{Reg}(H, A) \]
from the set of all left B-linear and right $H$-colinear isomorphisms $A \overset{\simeq}{\longrightarrow}
B \otimes H$ to the set of all $*$-invertible linear maps $H \to A$; this will immediately imply (a). 
To prove the claim, suppose $\xi(f) = F$, $\eta(g) = G$. We have the linear injection
\[ \operatorname{End}_{B-}^{-H}(B \otimes H) \to \operatorname{Hom}(H, A), \ p \mapsto 
[h \mapsto (\mathrm{id}_B \otimes \varepsilon)\circ p(1\otimes h))].\]
One sees that this injection preserves the unit, and sends the composite $g\circ f$ to $F * G$. 
Let 
\[ \operatorname{End}_{B-}^{-H}(A) \to \operatorname{Hom}(H, A) \]
be the linear injection obtained by composing the inclusion $\operatorname{End}_{B-}^{-H}(A)\hookrightarrow
\operatorname{End}_{B-}(A)$ with $\operatorname{Hom}_{B-}^{-H}(\beta_A^{-1}, A)$. One sees that 
this last injection preserves the unit, and sends $f\circ g$ to $G * F$. It follows that if $f$ and $g$ are
inverses of each other, then $F$ and $G$ are $*$-inverses of each other. This proves the claim. 

(a) $\Leftrightarrow$ (c). In general, a \emph{cleft system} (for a Hopf algebra $H$ over an algebra $B$) 
\cite[p.4539]{M} 
is a pair $(A, \phi)$ of an $H$-comodule algebra $A$ with $B = A^{\mathrm{co}H}$, together with a
section $\phi : H \to A$. Two such systems $(A_i, \phi_i)$, $i = 1,2$, are said to be isomorphic if there
exists an $H$-comodule algebra isomorphism $\zeta : A_1 \overset{\simeq}{\longrightarrow} A_2$
such that $\zeta|_B = \mathrm{id}_B$ and $\zeta \circ \phi_1 = \phi_2$.  (It is known that an $H$-comodule
algebra map $\zeta : A_1 \to A_2$ such that $\zeta|_B = \mathrm{id}_B$ is necessarily bijective; see 
\cite[Lemma 1.3]{M}, for example.)

There is a one-to-one correspondence
between the set of all isomorphic classes of the cleft systems $(A, \phi)$ and the set of all crossed systems 
$(\rightharpoonup, \sigma)$, both for $H$ over $B$; see \cite[Proposition 1.4]{M}. 
In fact, given $(A, \phi)$, the corresponding $(\rightharpoonup, \sigma)$ is given by 
\begin{equation}\label{construct_H-data}
\begin{aligned}
h \rightharpoonup b &= \phi(h_{(1)}) b\, \phi^{-1}(h_{(2)}),\\
\sigma(g,h) &= \phi(g_{(1)})\phi(h_{(1)})\phi^{-1}(g_{(2)}h_{(2)}),  
\end{aligned} 
\end{equation}
where $b\in B$, $g, h \in H$.
On the other hand, given $(\rightharpoonup, \sigma)$, the corresponding $(A, \phi)$ is given by
\[ A = B\rtimes_{\sigma} H, \quad \phi = 1 \otimes - : H \to B\rtimes_{\sigma} H,\ h \mapsto 1 \otimes h. \]
One sees that the cleft system thus given corresponds to the original $(\rightharpoonup, \sigma)$. 
One the other hand, if $(A, \phi) \mapsto (\rightharpoonup, \sigma)$, then 
$(B\rtimes_{\sigma} H, 1 \otimes -)$ is isomorphic to the original $(A, \phi)$ via
\[ B \rtimes_{\sigma} H \overset{\simeq}{\longrightarrow} A,\ b \otimes h \mapsto b\, \phi(h). \] 
The one-to-one correspondence thus shown proves the desired equivalence. 
\end{proof}

\begin{prop}\label{prop:smash_product} 
For an $H$-comodule algebra $A$ with $B = A^{\mathrm{co}H}$, 
TFAE:
\begin{itemize}
\item[(a)] There is an H-comodule algebra map $H \to A$; 
\item[(b)] There is an $H$-comodule algebra isomorphism $\alpha$ from $A$ to some $H$-smash product
$B \rtimes H$ such that $\alpha|_{B} = \mathrm{id}_B$. 
\end{itemize}
\end{prop} 
\begin{proof} 
One sees that if $(A, \phi)\leftrightarrow (\rightharpoonup, \sigma)$ under the one-to-one
correspondence given in the last proof, then $\phi$ is an $H$-comodule algebra map
if and only if $\sigma$ is trivial, i.e., $\sigma = \varepsilon \otimes \varepsilon$. 
This proves the desired equivalence. 
\end{proof} 

Given a Hopf algebra $H$, classify all cleft $H$-comodule algebras, or more generally, all $H$-Galois
extensions. Many authors already answered this question for various $H$; see \cite{M}, for example.

\section{Application~I---equivariant smoothness}\label{sec:smooth}

Very roughly speaking, spaces correspond to commutative algebras. Grothendieck formulated
the smoothness of spaces in terms of the corresponding commutative algebras. The definition,
with the commutativity removed, is as follows:
an algebra $R$ is said to be \emph{smooth}, if any algebra surjection  
onto $R$ with nilpotent kernel splits; see \cite[p.314]{Weibel}, for example. By using Hopf crossed products,
we prove a theorem 
which characterizes the
\emph{equivariant smoothness} of Hopf algebras. For an alternative, categorical approach, see \cite{AMS}. 

\begin{theorem}\label{thm:equivariant_smoothness}
For a Hopf algebra $H$, TFAE:
\begin{itemize}
\item[(a)] $H$ is hereditary as an algebra, or in other words, $\operatorname{gl.dim} H \le 1$, 
where $\operatorname{gl.dim} H$ denotes the global dimension of the algebra $H$;
\item[(b)] Given a surjection $C \twoheadrightarrow D$ of right (or left) $H$-comodule algebras with nilpotent kernel, 
every $H$-comodule algebra map $H \to D$ can lift to such a map $H \to C$; 
\item[(c)] Every surjection $A \twoheadrightarrow H$ of right (or left) $H$-comodule algebras onto $H$ with nilpotent
kernel splits. 
\end{itemize}
\end{theorem}

\begin{rem}\label{rem:gldim}
Let $H$ be a Hopf algebra. When we regard a vector space $V$ as a trivial left or right $H$-module
through $\varepsilon$, we write as ${}_{\mathrm{triv}}V$ or $V_{\mathrm{triv}}$. If $0 \leftarrow k
\leftarrow P_0 \leftarrow P_1 \leftarrow \dots$ is a projective resolution of the trivial $H$-module 
$k= {}_{\mathrm{triv}}k$, 
then $0 \leftarrow M = k \otimes M \leftarrow P_0\otimes M \leftarrow P_1 \otimes M \leftarrow \dots$
gives a projective resolution of any left $H$-module $M$, where the $H\otimes H$-modules $P_i \otimes M$
are regarded as $H$-modules through $\Delta : H \to H\otimes H$.
(To see this, note from the proof of \cite[Theorem 3.1.3, p.29]{Mon} that $F \otimes M$ is $H$-free, if $F$ is
a free left $H$-module.)   
It follows that $\operatorname{gl.dim} H$ coincides
with the projective dimension of $k$. Since there is an isomorphism between the standard complex for
computing $\operatorname{Ext}^{\bullet}_H(k, M)$ and the standard complex for computing the Hochschild
cohomology $HH^{\bullet}(H, M_{\mathrm{triv}})$ (see \cite[Lemma 9.1.9, p.306]{Weibel}), 
Condition (a) above is equivalent to
\begin{itemize}
\item[(a$'$)] $HH^2(H, M_{\mathrm{triv}}) = 0$ for all left $H$-modules $M$, 
\end{itemize}
where $M$ is regarded as an $H$-bimodule by the original left $H$-module structure and the
trivial right $H$-module structure. 
\end{rem}

To prove the theorem we need some preparations. In what follows, $H$ denotes a Hopf algebra.
Let $A$ be a \emph{cleft} $H$-comodule
algebra. Suppose that $A$ is augmented, i.e., is given a specific algebra map, 
$\varepsilon_A : A \to k$. 

\begin{lemma}\label{lem:augmented_section}
A section $\phi : H \to A$ can be chosen to be augmented, i.e., so as $\varepsilon = \phi\circ \varepsilon_A$.
\end{lemma} 
\begin{proof}
Replace an arbitrary section $\phi$ with $h \mapsto \varepsilon_A(\phi^{-1}(h_{(1)}))\phi(h_{(2)})$. 
\end{proof} 

\begin{lemma}\label{lem:augmented_data}
Let $A=(A, \varepsilon_A)$ be as above. Set $B = A^{\mathrm{co}H}$. Suppose that $\phi : H \to A$
be an augmented section. Then the corresponding crossed system 
\[ (\rightharpoonup \, : H \otimes B \to B,\ \sigma : H \otimes H \to B) \]
given by \eqref{construct_H-data} is augmented in the sense that 
\[ \varepsilon \otimes \varepsilon_A = 
\varepsilon_A\circ \rightharpoonup, \quad
\varepsilon \otimes \varepsilon = \varepsilon_A\circ \sigma. \]
\end{lemma}
\begin{proof}
This follows easily if one notices that $\phi^{-1}$ is augmented if $\phi$ is.  
\end{proof}

Let $B=(B, \varepsilon_B)$ be an augmented algebra. An augmented cleft $H$-comodule algebra 
$A=(A, \varepsilon_A)$ such that 
\[ A^{\mathrm{co}H} = B, \quad \varepsilon_A|_{B} = \varepsilon_B \] 
is called an \emph{augmented} $H$-\emph{cleft extension over} $B$. 
Two such extensions $A_i = (A_i, \varepsilon_{A_i})$, $i = 1,2$, are said to be \emph{isomorphic}, 
if there is an $H$-comodule algebra isomorphism $\zeta : A_1 \overset{\simeq}{\longrightarrow} A_2$ 
such that $\zeta|_B = \mathrm{id}_B$ and $\varepsilon_1 = \varepsilon_2 \circ \zeta$. Let 
\[ \operatorname{Cleft}_{\varepsilon}(H; B) \]
denote the set of the isomorphism classes $[A]$ of all augmented $H$-cleft extensions $A$ over $B$. 

\begin{corollary}\label{cor:augmented_crossed_product}
Every element of $\operatorname{Cleft}_{\varepsilon}(H; B)$ is of the form $[B \rtimes_{\sigma} H]$, where
$B \rtimes_{\sigma} H$ is the $H$-crossed product associated with some augmented 
crossed system $(\rightharpoonup, \sigma)$,
and its augmentation is given by $\varepsilon_B \otimes \varepsilon : B \rtimes_{\sigma} H \to k$. 
\end{corollary}
\begin{proof}
This follows from the preceding two lemmas. 
\end{proof}

Let $B=(B, \varepsilon_B)$ be as above. Set
$B^+ = \operatorname{Ker}\varepsilon_B$, the augmentation ideal. Note $B = k \oplus B^+$. 

\begin{lemma}\label{2-cocycle}
Assume $(B^+)^2=0$. Given a linear maps
\[ \triangleright : H\otimes B^+ \to B^+, \quad s : H \otimes H \to B, \]
define linear maps $\rightharpoonup : H \otimes B \to B$, $\sigma : H \otimes H \to B$ by
\begin{align}
h \rightharpoonup (c + b) &= \varepsilon(h) c + h \triangleright  b, \quad h \in H, c\in k, b \in B^+, \label{hit}\\
\sigma (g, h) &= \varepsilon(g)\varepsilon(h)1 + s(g, h), \quad g, h \in H. \label{sigma} 
\end{align}

(1)~~Necessarily, $\rightharpoonup$ gives a measuring, and $\sigma$ is $*$-invertible.

(2)~~$\rightharpoonup$ and $\sigma$
form a crossed system if and only if $B^+$ is 
a left $H$-module under $\triangleright$,
and $s$ is a normalized Hochschild 2-cocycle with coefficients in the $H$-bimodule 
$B^+ = {}_{\triangleright}B^+_{\mathrm{triv}}$, which is a left $H$-module under $\triangleright$, and
is the trivial right $H$-module. 
\end{lemma}
\begin{proof}
(1) Easy to see. Note that $\sigma^{-1}$ is given by $(g,h) \mapsto \varepsilon(g)\varepsilon(h)1 - s(g, h)$. 

(2) One sees that the condition \eqref{H-data2} is equivalent to 
the associativity of the action $\triangleright$, and the condition \eqref{H-data1} is equivalent to 
the unitality of the action $\triangleright$, 
and the normalization condition
\begin{equation}\label{Hochschild1}
s(h, 1) = 0 = s(1, h), \quad h \in H. 
\end{equation}
Recall that $s$ is called a normalized Hochschild 2-cocycle \cite[p.301]{Weibel} with 
coefficients in ${}_{\triangleright}B^+_{\mathrm{triv}}$, if
\eqref{Hochschild1} and the condition 
\begin{equation}\label{Hochschild2}
g \triangleright s(h, \ell ) + s(g, h\ell ) = s(g, h) \varepsilon ( \ell ) + s(gh, \ell ),\quad g, h, \ell \in H 
\end{equation} 
are satisfied. Let $g, h, \ell \in H$. Writing  as $\Delta(g) = \sum_{i}g_i \otimes g'_i$, 
$\Delta(h) = \sum_{j}h_j \otimes h'_j$, $\Delta(\ell ) = \sum_{k}\ell_k \otimes \ell'_k$, we compute 
\begin{align*}
\text{LHS of}~\eqref{H-data3} &= \sum_{i,j,k} [\varepsilon(g_i)\varepsilon(h_j)\varepsilon(\ell_k) +
g_i \triangleright s(h_j, \ell_k)]\, [\varepsilon(g'_i)\varepsilon(h'_j)\varepsilon(\ell'_k)+s(g'_i, h'_j\ell'_k)]\\
&= \varepsilon(g)\varepsilon(h)\varepsilon(\ell) + g\triangleright s(h, \ell) + s(gh, \ell), \\
\text{RHS of}~\eqref{H-data3} &= \sum_{i,j} [\varepsilon(g_i)\varepsilon(h_j) +
s(g_i, h_j)]\, [\varepsilon(g'_i)\varepsilon(h'_j)\varepsilon(\ell)+s(g'_ih'_j, \ell)]\\
&= \varepsilon(g)\varepsilon(h)\varepsilon(\ell) + s(g, h) \varepsilon(\ell) + s(gh, \ell).
\end{align*} 
It follows that \eqref{H-data3} is equivalent to \eqref{Hochschild2}. 
\end{proof}

\begin{prop}\label{prop:bijection}
Let $B$ be an augmented algebra such that $(B^+)^2 = 0$. By Lemma \ref{2-cocycle}, 
every pair $(\triangleright, s)$ of a left $H$-module structure $\triangleright$ on $B^+$ and a normalized
Hochschild 2-cocycle $s$ with coefficients in ${}_{\triangleright}B^+_{\mathrm{triv}}$ gives an element 
$[B \rtimes_{\sigma} H]$ in $\operatorname{Cleft}_{\varepsilon}(H; B)$, where $B \rtimes_{\sigma} H$
denotes the augmented cleft $H$-comodule algebra which is associated with $(\rightharpoonup, \sigma)$ 
defined by \eqref{hit}, \eqref{sigma}. The assignment $(\triangleright, s)\mapsto [B \rtimes_{\sigma} H]$
induces a bijection of sets,
\[ \bigsqcup_{\triangleright}\, HH^2(H, {}_{\triangleright}B^+_{\mathrm{triv}}) 
\overset{\simeq}{\longrightarrow} \operatorname{Cleft}_{\varepsilon}(H; B), \]
where $\bigsqcup_{\triangleright}$ denotes
the disjoint union with respect to all left $H$-module structures $\triangleright$ on $B^+$. 
\end{prop}
\begin{proof}
Suppose two pairs $(\triangleright, s)$, $(\triangleright', s')$ give $B \rtimes_{\sigma} H$, $B \rtimes_{\sigma'} H$,
respectively. A unit-preserving, augmented, left $B$-linear and right $H$-colinear map
$B \rtimes_{\sigma} H\to B \rtimes_{\sigma'} H$ 
is of the form
\[ f_t(b\otimes h) = \sum_j\, b(\varepsilon(h_j) + t(h_j))\otimes h_j', \quad b\in B, h\in H, \]
where $t : H \to B^+$ is a linear map with $t(1) = 0$, and we have written as 
$\Delta(h) = \sum_j\, h_j\otimes h'_j$. This $f_t$ is bijective with inverse $f_{-t}$. One sees that $f_t$
is right $B$-linear and satisfies $f_t((1\otimes g)(1\otimes h))=f(1\otimes g)f(1\otimes h)$ for all 
$g, h \in H$ if and only if $\triangleright = \triangleright'$, and $s - s'$ equals the coboundary $\partial t$
of $t$, i.e., 
\[ s(g, h) - s'(g, h) = g \triangleright t(h) - t(gh) + t(g)\varepsilon(h), \quad g, h \in H. \]
This shows that the map given in the proposition is well defined and injective. It is surjective
by Corollary \ref{cor:augmented_crossed_product}. 
\end{proof}

\begin{rem}\label{rem:split_extension}
Let $B$ be an augmented algebra. An augmented $H$-cleft extension over $B$ is said to be 
\emph{split}, if there is an augmented $H$-comodule algebra map $H \to A$, which is
necessarily $*$-invertible. Split extensions are closed under isomorphism. 
Assume $(B^+)^2 = 0$. Then one sees that through the bijection
obtained in the preceding proposition, the zero cohomology classes correspond precisely to the
isomorphism classes of all split extensions. It follows from Remark \ref{rem:gldim}
that Condition (a) of Theorem \ref{thm:equivariant_smoothness}
is equivalent to
\begin{itemize}
\item[(a$''$)] For any augmented algebra $B$ with $(B^+)^2=0$, an augmented $H$-cleft extension 
over $B$ is necessarily split.
\end{itemize}
\end{rem}

\begin{proof}[Proof of Theorem \ref{thm:equivariant_smoothness}]
Let $\pi : A \twoheadrightarrow H$ be a surjection as in (c). 

Assume (b). Then the identity map on $H$ can lift to $H \to A$. This last map is a splitting required by (c).
Therefore, (b) $\Rightarrow$ (c). 

Keep $\pi$ as above. 
We claim that $\pi$ splits $H$-colinearly. To see this, set $I = \operatorname{Ker} \pi$.
We have an integer $n > 0$ such that $I^n = 0$. We have the sequence
\[ A = A/I^n \twoheadrightarrow A/I^{n-1} \twoheadrightarrow \dots \twoheadrightarrow A/I = H \]
of surjections in the category $\mathcal{M}_A^H$ of $(H,A)$-Hopf modules; see Remark \ref{rem:Hopf_module}
below. Fix $0 < i \le n$. Then we have the short exact sequence
\[ 0 \to I^{i-1}/I^{i} \to A/I^{i} \to A/I^{i-1} \to 0 \]
in $\mathcal{M}_A^H$. Annihilated by $I$, the kernel $I^{i-1}/I^{i}$ is in $\mathcal{M}_{A/I}^H = \mathcal{M}_H^H$.
Since the Hopf module theorem shows that the kernel is cofree as an $H$-comodule, the short
exact sequence above splits $H$-colinearly. This proves the claim. 

Let $\phi : H \to A$ be an $H$-colinear splitting of $\pi$. We see that $\phi$ is $*$-invertible in 
$\operatorname{Hom}(H, A)$, since it is so modulo the nilpotent ideal $\operatorname{Hom}(H, I)$. 
By the same replacement as given in \eqref{replacement}, 
we may suppose $\phi(1) = 1$. Set $B = A^{\mathrm{co}H}$, and suppose that it is augmented by
$\pi|_B : B \to H^{\mathrm{co}H} = k$. Then $A$ is an augmented $H$-cleft extension
over $B$, with augmentation $\varepsilon_A := \varepsilon\circ \pi$. Conversely, an augmented $H$-cleft
extension $B \rtimes_{\sigma} H$ over any augmented algebra $B$ gives an $H$-comodule algebra surjection
$\varepsilon_B \otimes \mathrm{id}_H : B \rtimes_{\sigma} H \to H$.  
Note that for any integer $n>0$, $(\operatorname{Ker} \pi)^n = 0$ if and only if
$(B^+)^n =0$. 
Note also that $\pi$ splits as an $H$-comodule algebra map 
if and only if $A$ is split as an augmented $H$-cleft extension. 

We see from Remark \ref{rem:split_extension} that 
Condition (a) is equivalent to 
\begin{itemize}
\item[(c$'$)] Every surjection $\pi : A \twoheadrightarrow H$ of right $H$-comodule algebras such that
$(\operatorname{Ker} \pi)^2 = 0$ splits. 
\end{itemize}
Since (c) obviously implies (c$'$), it remains to prove (c$'$) $\Rightarrow$ (b). 

Assume (c$'$). Let $\varpi : C \to D$ be a surjection of non-zero $H$-comodule algebras with 
$J :=\operatorname{Ker}\varpi$ nilpotent. Let $\psi : H \to D$ be an $H$-comodule algebra map,
which is necessarily $*$-invertible. To see that $\psi$ can lift to some $H \to C$, we may suppose
$J^2 = 0$, by replacing $\varpi$ with $C/J^{2^i} \to C/J^{2^{i-1}}$, $i = 1,2,\dots$ 
Note that $\psi$ is injective, since it is identified with
$H \to D^{\mathrm{co}H} \rtimes H$, $h \mapsto 1 \otimes h$ through the isomorphism 
$D^{\mathrm{co}H} \rtimes H\overset{\simeq}{\longrightarrow} D$, $x \otimes h \mapsto x\psi(h)$. 
Construct from $\varphi$, $\psi$ the pull-back diagram of $H$-comodule algebras:
\begin{equation*}
\begin{CD}
A @>{\pi}>> H \\
@V{\varphi}VV @VV{\psi}V \\
C @>{\varpi}>> D
\end{CD}
\end{equation*}
(Regarding $H \subset D$ via the injection $\psi$, one may suppose $A = \varpi^{-1}(H)$.) 
Note that $\varphi$ is  an injection, and $\pi$ is a surjection
with $(\operatorname{Ker}\pi)^2 = 0$. By (c$'$), we have a splitting
$\phi$ of $\pi$. The composite $\varphi \circ \phi$ gives a desired lift of $\psi$. 
\end{proof}

\begin{rem}\label{rem:Hopf_module}
Let $A =(A, \rho_A)$ be an $H$-comodule algebra. An $(H, A)$-\emph{Hopf
module} \cite[p.144]{Mon} 
is a right $A$-module $M$ given a right $H$-comodule structure $\rho_M : M \to M \otimes H$
such that $\rho_M(ma) = \rho_M(m)\rho_A(a)$, $m \in M$, $a \in A$. All $(H, A)$-Hopf modules together with
$A$-linear and $H$-colinear maps form a category $\mathcal{M}^H_A$. Since $H = (H, \Delta)$ is an 
$H$-comodule algebra, we have the category $\mathcal{M}^H_H$. For every vector space $V$, we have 
$(V \otimes H, \mathrm{id}_V \otimes \Delta) \in \mathcal{M}^H_H$. Given $M \in \mathcal{M}^H_H$, we have a
morphism in $\mathcal{M}^H_H$, 
\[ M^{\mathrm{co}H} \otimes H \to M,\ m \otimes h \mapsto mh, \]
where $M^{\mathrm{co}H}$ is the subspace of $M$ defined just as in \eqref{coinvariants}. 
The Hopf module theorem \cite[Theorem 4.1.1, p.84]{Sw} states 
that this map is an isomorphism, and so that every object in 
$\mathcal{M}^H_H$ is cofree as an $H$-comodule, i.e., is the direct sum of some copies of $H$. 
\end{rem}

Given a coalgebra $C$, let $\operatorname{Corad} C$ denote the coradical of $C$, i.e., the (direct) sum
of all simple subcoalgebras of $C$; see \cite[p.181]{Sw}.  Here we state without  proof  the dual result of Theorem 
\ref{thm:equivariant_smoothness}, which has a nicer form with the nilpotent kernel condition
replaced by including coradicals. 

\begin{theorem}[\cite{MO}]\label{thm:dual_result}
For a Hopf algebra $H$, TFAE:
\begin{itemize}
\item[(a)] The global dimension of the coalgebra $H$ is at most 1, or equivalently, 
every quotient $H$-comodule of an arbitrary injective right or left $H$-comodule is injective;
\item[(b)] Given an inclusion $D \hookrightarrow C$ of left or right $H$-module coalgebras such that 
$D \supset \operatorname{Corad} C$, 
every $H$-module coalgebra map $D \to H$ can extend to such a map $C \to H$; 
\item[(c)] An inclusion $H \hookrightarrow C$ of $H$ into a left or right $H$-module coalgebra $C$
such that $H \supset \operatorname{Corad} C$ necessarily splits. 
\end{itemize}
\end{theorem}

This implies the following corollary, since a cosemisimple coalgebra is the same as a coalgebra
with global dimension zero.  

\begin{corollary}[\cite{Ma}]\label{cor:weak_Chevalley}
Suppose that 
$C$ is a Hopf algebra whose coradical $H:= \operatorname{Corad} C$ is a Hopf subalgebra; this is satisfied 
if $C$ is a pointed Hopf algebra \cite[p.157]{Sw}. 
Then 
the inclusion $H \hookrightarrow C$ splits as a left (or right) $H$-module coalgebra map, so that 
the left $H$-module coalgebra $C$ is isomorphic to a smash coproduct $H\cmdblackltimes C/H^+C$. 
\end{corollary} 

\begin{rem}\label{rem:module_coalgebra}
A \emph{left} (or \emph{right}) $H$-\emph{module} (resp., $H$-\emph{comodule}) \emph{coalgebra} is a 
left (or right) $H$-module
(resp., $H$-comodule) and coalgebra $D = (D, \Delta, \varepsilon)$ such that $\Delta : D \to D \otimes D$, 
$\varepsilon : D \to k$ are both $H$-linear (resp., $H$-colinear), where the $H \otimes H$-module
(resp., $H\otimes H$-comodule) $D \otimes D$ is regarded as an $H$-module (resp., $H$-comodule) 
through the coproduct $\Delta$ 
(resp., along the product) on $H$,
and $k$ is regarded as the trivial $H$-module (resp., $H$-comodule). If $D$ is a right $H$-comodule coalgebra
with $H$-comodule structure $\rho_D : D \to D\otimes H$, $\rho_D(d) = d_{(0)}\otimes d_{(1)}$, then one constructs 
a left $H$-module coalgebra on the left $H$-module $H \otimes D$ by defining
\begin{align*}
\Delta(h \otimes d) &= [h_{(1)} \otimes (d_{(1)})_{(0)}] \otimes [h_{(2)}(d_{(1)})_{(1)} \otimes d_{(2)}],\\  
\varepsilon(h \otimes d) &= \varepsilon(h)\, \varepsilon(d) 
\end{align*}
for $h \in H$, $d \in D$. This left $H$-module coalgebra is denoted by $H\cmdblackltimes D$, and is
called the \emph{smash coproduct} associated with $\rho_D$. 
\end{rem}

\section{Application~II---tensor product decomposition in super-commutative Hopf superalgebras}
\label{sec:tensor_product_decomposition}

In what follows we assume that the characteristic $\operatorname{char} k$ of $k$ is not 2. 
The adjective $``$super" is a synonym of $``\mathbb{Z}/(2)$-graded."
Thus, a \emph{super-vector space} is a vector space, $V = V_0 \oplus V_1$, graded by 
the group $\mathbb{Z}/(2)=\{0,1\}$ of order 2. Elements $v, w, \dots$ in $V$ will be often supposed to 
be homogeneous and non-zero without explicit citation. 
In that case, $|v|$ denotes the degree of $v$, and $v$ is said to be 
\emph{even} (resp., \emph{odd}) if $|v| = 0$ (resp., $|v| = 1$).  
A \emph{superalgebra} is a $\mathbb{Z}/(2)$-graded algebra. 

The trivial symmetry 
$v \otimes w \mapsto w\otimes v$ on the tensor product of vector spaces is replaced by 
the \emph{supersymmetry} 
\[ c = c_{V,W} : V\otimes W \overset{\simeq}{\longrightarrow} W \otimes V,\ 
c(v \otimes w)= (-1)^{|v||w|} w \otimes v \]
on the tensor product on super-vector spaces $V$, $W$; 
note $c_{W,V} \circ c_{V,W} = \mathrm{id}_{V \otimes W}$.
Our assumption $\operatorname{char} k\ne 2$
ensures that this symmetry is not trivial. 
A superalgebra $A$ with product $\mu_A : A \otimes A \to A$ is said to be \emph{super-commutative}
if $\mu_A = \mu_A\circ c_{A,A}$, or explicitly if $ab = (-1)^{|a||b|}ba$,\ $a, b \in A$. If this is the case,
$A$ includes $A_0$ as a central subalgebra, and $a^2=0$ if $a\in A_1$. 
 The tensor product $\otimes$ will be denoted 
by $\underline{\otimes}$, if the supersymmetry effects on it. Thus, given superalgebras $A$, $B$
(with products $\mu_A$, $\mu_B$, respectively), we have the superalgebra $A\, \underline{\otimes}\, B$
whose product is given by
\[ (\mu_A \otimes \mu_B)\circ (\mathrm{id}_A \otimes c_{B,A} \otimes \mathrm{id}_B) : 
(A\otimes B)\otimes (A\otimes B) \to A \otimes B, \]
or more explicitly by
\[ (a \otimes b)(a' \otimes b') = (-1)^{|b||a'|} aa' \otimes bb'. \]
A \emph{super-bialgebra} is a superalgebra and coalgebra $A=(A,\Delta, \varepsilon)$ such that
$\Delta : A \to A\, \underline{\otimes}\, A$, $\varepsilon ; A\to k$ are both superalgebra maps. 
Suppose that the identity map $\mathrm{id}_A$
on a super-bialgebra $A$ has a $*$-inverse $S$ in $\operatorname{Hom}(A,A)$; it is called an \emph{antipode}
of $A$.  Then $S$ necessarily preserves the $\mathbb{Z}/(2)$-grading, and satisfies
\[  S\circ \mu_A = \mu_A\circ (S \otimes S)\circ c_{A,A},\quad 
\Delta \circ S = c_{A,A}\circ (S \otimes S)\circ \Delta. \]
A \emph{Hopf superalgebra} is a super-bialgebra with antipode. 

Super-objects $A$, such as super-vector spaces or (Hopf) superalgebras, are said to be \emph{purely
even}, if $A_0 = A$, $A_1 = \{0\}$. These are the same as ordinary objects, such as vector spaces or
(Hopf) algebras. In other words, ordinary objects, as purely even super-objects, form a special class
of super-objects. 

The linear dual $V^*= \operatorname{Hom}(V, k)$ of a super-vector space $V$ is a super-vector space with 
$(V^*)_i =(V_i)^*$, $i=0,1$. If $A$ is a finite-dimensional Hopf superalgebra, the super-vector space $A^*$ 
is a Hopf superalgebra with respect to the dual algebra (resp., coalgebra) structure of the coalgebra (resp.,
algebra) $A$. 

\begin{example}
Let $V$ be a super-vector space which is purely odd in the sense  $V_0 = \{0\}$, $V_1 = V$. 
The exterior algebra $A= \wedge(V)$ 
forms a Hopf superalgebra in which every element $v$ in $V$ is supposed to be \emph{primitive}, i.e.,
\[ \Delta(v) = 1 \otimes v + v \otimes 1, \quad \varepsilon(v) = 0. \] 
This is super-commutative, and is super-cocommutative
in the sense $\Delta = c_{A,A}\circ \Delta$. Suppose $\dim V < \infty$. The bilinear maps 
$\langle~~,~\rangle : \wedge^n(V^*) \times \wedge^n(V) \to k$, $n=1, 2,\dots$ given by
\[ \langle f_1\wedge \dots \wedge f_n, v_1 \wedge \dots \wedge v_n \rangle =
\sum_{\sigma \in \mathfrak{S}_n} (\mathrm{sgn} \, \sigma) \, f_1(v_{\sigma(1)})
\dots f_n(v_{\sigma(n)}) \]
amount to a bilinear map 
\[ \langle~~,~\rangle : \wedge(V^*) \times \wedge(V) \to k \]
such that $\langle 1, 1 \rangle = 1$, and $\langle \wedge^m(V^*), \wedge^n(V)\rangle = 0$ unless $m=n$.  
This last bilinear map is non-degenerate, and gives rise to an isomorphism of Hopf superalgebras,
\begin{equation}\label{duality}
\wedge(V^*) \overset{\simeq}{\longrightarrow} (\wedge(V))^*. 
\end{equation}
\end{example} 

Just as commutative Hopf algebras correspond to affine group schemes (see \cite[Chapter 1]{Waterhouse}), super-commutative
Hopf superalgebras correspond to \emph{affine supergroup schemes} \cite[Section~11.1]{CCF}.  
Let $A$ be a super-commutative
Hopf superalgebra. The corresponding \emph{affine supergroup scheme} $\operatorname{SSp} A$
is the group-valued functor 
defined on the category $\mathbf{SAlg}_k$ of super-commutative superalgebras
over $k$, which associates to each $R \in \mathbf{SAlg}_k$, the group $\mathbf{SAlg}_k(A, R)$
of all superalgebra maps $A \to R$; this is indeed a group with respect to the
$*$-product given by \eqref{*-product}. The one-sided (necessarily, two-sided) ideal 
\[ AA_1=A_1^2 \oplus A_1 \, (= A_1A) \]
generated by $A_1$ 
is a super-ideal of $A$, and the quotient superalgebra
\begin{equation}\label{H}
H := A/A_1 = A_0/A_1^2
\end{equation}
is a purely even Hopf superalgebra which represents the subgroup functor $R \mapsto \mathbf{SAlg}_k(A, R_0)$
of $\operatorname{SSp}A$. Just for affine group schemes, the cotangent super-vector space of 
at unity is given by $A^+/(A^+)^2$, where $A^+ = \operatorname{Ker} \varepsilon$. 
We denote 1-component in $A^+/(A^+)^2$ by
\begin{equation}\label{W}
W := (A^+/(A^+)^2)_1 = A_1/A_0^+A_1, 
\end{equation}
where $A_0^+ = A_0 \cap A^+$. 
Let $\pi : A \to H$ denote the quotient Hopf superalgebra map. 
Note that $H$ is an ordinary Hopf algebra, and the superalgebra $A$ 
is an $H$-comodule algebra
by the superalgebra map $(\mathrm{id}_A\otimes \pi)\circ \Delta : A \to A \otimes H$. 
  
\begin{theorem}[\cite{M1}]\label{thm:tensor_product_decomposition}
Let $A$ be a super-commutative Hopf superalgebra, and let $H$, $W$ be as in \eqref{H}, \eqref{W}. Then there
is a right $H$-colinear isomorphism of superalgebras  
\[ \alpha : A \overset{\simeq}{\longrightarrow} \wedge(W) \otimes H \]
which is augmented, i.e., $\varepsilon = (\varepsilon\otimes \varepsilon)\circ \alpha$. 
\end{theorem}

This result pays an essential role to prove a fundamental result on quotients of affine algebraic 
supergroup schemes (see \cite[Theorem 0.1]{MZ}), and to prove a category equivalence between the finitely generated
super-commutative Hopf superalgebras and the Harish-Chandra pairs (see \cite[Theorem~4.23]{M2}).

\begin{proof}
To prove this, 
we assume that $A$ is finitely generated as an algebra over the central subalgebra $A_0$.
The result in this special case, applied to an argument using the inductive limit and the Zorn
Lemma, proves the result in general; see \cite[pp.302--303]{M1}. It is known that the assumption above 
is equivalent to that $\dim W< \infty$, and the kernel 
\[ A_1^2 \oplus A_1 = \operatorname{Ker}( \pi : A \to H ) \]
of the quotient map is nilpotent; see \cite[Proposition 4.4]{M1}.     

\vspace{3mm}
\emph{Step 1}.  
Set $B = A^{\mathrm{co}H}$; this is a super-subalgebra of $A$. 
Suppose that $B$ is augmented by the restriction $\varepsilon |_B : B \to k$ of the counit $\varepsilon$ on $A$. 
Write $B^+ = A^+ \cap B$, $B_0^+ = A^+ \cap B_0$.   
Then one defines, just as the $W$ for $A$,  
\[ W_B = B_1/B_0^+B_1. \]
We claim 
\begin{itemize}
\item[(a)] $B_0^+ = B_1^2$;
\item[(b)] $W_B = B^+/(B^+)^2$, or in other words, $B_0^+/((B_0^+)^2 + B_1^2) =0$; 
\item[(c)] The inclusion $B \hookrightarrow A$ of augmented superalgebras induces an
isomorphism, $W_B \overset{\simeq}{\longrightarrow} W$. 
\end{itemize}

Let $\pi_0 : A_0 \to A_0/A_1^2 =H$ denote the quotient map.  
Note that $A_0$ is an $H$-comodule subalgebra of $A$, and $\pi_0 = \pi|_{A_0}$. 
Since the kernel 
$\operatorname{Ker}\pi_0$ of $\pi_0$ is nilpotent, we obtain, as in the proof of Theorem 
\ref{thm:equivariant_smoothness}, a unit-preserving $H$-colinear splitting $\phi : H \to A_0$ 
of $\pi_0$, which is necessarily $*$-invertible. 
Regard this $\phi$ as a map $H \to A$; it is a unit-preserving 
$H$-colinear splitting of $\pi$, which is $*$-invertible and $\mathbb{Z}/(2)$-graded. 
Recall from the proof of Theorem \ref{thm:DT} that the crossed system
$(\rightharpoonup, \sigma)$ given by \eqref{construct_H-data}
constructs an $H$-crossed product $B\rtimes_{\sigma}H$,   
and we have an $H$-comodule algebra isomorphism which is identical on $B$, 
\begin{equation}\label{isom}
B \rtimes_{\sigma} H \overset{\simeq}{\longrightarrow} A, \ b \otimes h \mapsto b\, \phi(h), 
\end{equation}
Note that $h \rightharpoonup b = \varepsilon(h)\, b$ for all $h \in H$, $b \in B$,   
and the last isomorphism preserves the $\mathbb{Z}/(2)$-grading, where $(B \rtimes_{\sigma} H)_i
= B_i \otimes H$, $i = 0, 1$. Since $\pi(b) = \varepsilon(b)1$ for all $b \in B$, the isomorphism \eqref{isom}
induces
\[ \frac{B \rtimes_{\sigma} H}{B^+ \otimes H} \overset{\simeq}{\longrightarrow} H. \]
Therefore, $B^+ \otimes H$ in $B \rtimes_{\sigma} H$ coincides with the (right) ideal generated by $B_1 \otimes H$,
whence we have 
\[ B^+ \otimes H = B_1B \otimes H. \] 
Hence, $B^+ = B_1B~(= B_1^2 \oplus B_1)$, which proves (a). Obviously, (b) follows from (a). 

We identify $A = B \rtimes_{\sigma} H$ through the isomorphism \eqref{isom}. 
Since $\phi$ preserves the counit, the counit $\varepsilon$ on $A$ is translated to  
$\varepsilon \otimes \varepsilon : B \rtimes_{\sigma} H\to k$. We see 
that on $A = B \rtimes_{\sigma} H$, 
\[ A_0^+A_1 = (B_0^+ \otimes H  + B_0 \otimes H^+)(B_1 \otimes H) 
= B_0^+B_1 \otimes H + B_1\otimes H^+, \] 
since $\sigma(H^+ \otimes H) \subset B_0^+$, and so
\[  (B_0 \otimes H^+)(B_1 \otimes H) 
\subset B_0^+B_1 \otimes H + B_1\otimes H^+. \]              
It follows that $ W = W_B \otimes k$, which proves (c).  

\vspace{3mm}
\emph{Step 2}. We claim that there exists an augmented isomorphism $\gamma :
B \overset{\simeq}{\longrightarrow} \wedge(W)$ of superalgebras. Let $A^{\circ}$ be the
\emph{dual coalgebra} \cite[Sect.6.0]{Sw} of the algebra $A$. 
Note that the dual algebra $A^*$ of the coalgebra $A$ is a superalgebra. 
Recall from \cite[Proposition 6.0.3, p.115]{Sw} that $A^{\circ}$ is the pull-back of the super-subalgebra
$A^*\,\underline{\otimes}\, A^*$ of the superalgebra $(A \, \underline{\otimes}\, A)^*$ along the dual $\mu^* : A^* \to 
(A \,\underline{\otimes}\, A)^*$ of the product $\mu : A \,\underline{\otimes}\, A \to A$ on $A$. Since $\mu^*$
is a superalgebra map, one sees that $A^{\circ}$ is a superalgebra, and is, moreover, a
Hopf superalgebra. We have a canonical map
$A \to (A^{\circ})^*$
of superalgebras. 
Analogously to the ordinary situation (see \cite[Section 12.2]{Waterhouse}), 
the vector space 
\[ U := \{ u \in (A^{\circ})_1 \mid \Delta(u) = 1 \otimes u + u \otimes 1 \} \] 
of all odd primitives in $A^{\circ}$ is dual to $W$, i.e., $U = W^*$, which gives rise to a canonical
isomorphism $(\wedge(U))^* \simeq \wedge(W)$ as in \eqref{duality}. See \cite[Proposition 4.3(2)]{M1}
for more details. 

Choose an arbitrary basis $u_1, u_2, \dots u_n$ of $U$, and let $\iota : \wedge(U) \to A^{\circ}$ 
be the linear map defined on the natural basis of $\wedge(U)$ so that
\[ \iota(1) =1, \quad \iota(u_{i_1}\wedge u_{i_2} \wedge \dots \wedge u_{i_r}) = u_{i_1}u_{i_2}\dots u_{i_r},\]
where 
$1\le i_1< i_2 <\dots <i_r \le n$, $1\le r \le n$.  Let
\[ \delta : A \to (A^{\circ})^* \overset{\iota^*}{\longrightarrow} (\wedge(U))^* \overset{\simeq}{\longrightarrow} 
\wedge(W) \]
be the composite of the dual $\iota^*$ of $\iota$ with the canonical map and isomorphism above. 
Define 
$\gamma = \delta|_{B} : B \to \wedge(W)$. Since $\iota$ is a unit-preserving super-coalgebra map, 
$\gamma$ and $\delta$ are both augmented superalgebra map. To prove the claim, it remains 
to show that $\gamma$ is an isomorphism. For this, it is enough to show that 
the graded algebra map
\[ \operatorname{gr} \gamma : \operatorname{gr}B = \bigoplus_{i=0}^{\infty}(B^+)^i/(B^+)^{i+1}
\to \operatorname{gr} (\wedge(W)) = \wedge(W), \]
between the graded algebras associated with the filtrations given by powers of the augmentation ideals, 
is an isomorphism. Obviously, $(\operatorname{gr}\gamma)_0$ is the identity map on $k$.
The results (b), (c) in Step 1 show that $(\operatorname{gr}\gamma)_1 : B^+/(B^+)^2 \to W$ is isomorphic. 
Since $\operatorname{gr}B$ is super-commutative, there uniquely exists a graded 
algebra map $\wedge(W)\to \operatorname{gr}B$
which extends $(\operatorname{gr}\gamma)_1^{-1}$.
This last graded algebra map is seen to be an inverse of $\operatorname{gr}\gamma$, since  
$\operatorname{gr}B$ is generated by $(\operatorname{gr}B)_1$. 

\vspace{3mm}
\emph{Step 3}. 
Identify $B = \wedge(W)$ through $\gamma$ obtained above. Note that $B^+$ is nilpotent. Define
$\alpha : A \to \wedge(W) \otimes H$ by
\[ \alpha(a) = \delta(a_{(1)}) \otimes \pi(a_{(2)}),\quad a \in A. \]
We see that $\alpha$ is an augmented,
right $H$-colinear superalgebra map.
Moreover, $\alpha$ is a left $B$-linear map to a free $B$-module,
which is an isomorphism modulo the nilpotent ideal $B^+$, as is seen from the isomorphism \eqref{isom}.
Hence, $\alpha$ is an isomorphism as desired.  
\end{proof}

\section*{Acknowledgments}
These notes are based on the author's talks at \emph{Conference on Ring Theory} 
held at the Sun Yat-Sen University, Guangzhou in P.~R.~China, 
June 24--30, 2012.
The author thanks Professor Xiaolong Jiang for her very kind hospitality. 
He also thanks Professor George Szeto and Professor Sh\^{u}ichi Ikehata
for their friendly support.

\end{document}